\documentclass[12pt,english]{article}
\usepackage[letterpaper]{geometry}
\usepackage[utf8]{inputenc}
\usepackage{amsmath}
\usepackage{amsthm}
\usepackage{amssymb}
\usepackage{authblk}
\usepackage{hyperref}
\usepackage{tikz}

\providecommand{\keywords}[1]{\textbf{\textit{Keywords:}} #1}
\providecommand{\msc}[1]{\textbf{\textit{MSC2020:}} #1}

\newcommand{\R}{\mathbb R}
\newcommand{\B}{\mathcal B}

\usepackage{float}
\usepackage{babel,blindtext}
\usepackage{graphicx}

\newtheorem{theorem}{Theorem}[section]

\newtheorem{proposition}[theorem]{Proposition}
\newtheorem{definition}[theorem]{Definition}

\title{B\'ezout coefficients of coprime numbers approximate quadratic B\'ezier curves}
\author[1]{\sc{Benjam\'in A.~Itz\'a-Ortiz} \thanks{Corresponding author: {\tt itza@uaeh.edu.mx}}} 
\author[1]{\sc{Roberto L\'opez-Hern\'andez}\footnote{\tt roberto\_lopez@uaeh.edu.mx }}
\author[2]{\sc{Pedro Miramontes}\footnote{\tt pmv@ciencias.unam.mx }}

\affil[1]{{\small{Centro de Investigaci\'on en Matem\'aticas, Universidad Aut\'onoma del Estado de Hidalgo, Pachuca, Hidalgo, \sc{Mexico}}}}

\affil[2]{{\small{Facultad de Ciencias, Universidad Nacional Aut\'onoma de M\'exico, Mexico City, \sc{Mexico}}}}

\date{}
\begin{document}

\maketitle

\begin{abstract}
   Given a point $(p,q)$ with  nonnegative integer coordinates and $p\not=q$,  we prove that the quadratic B\'ezier curve relative to the points $(p,q)$, $(0,0)$ and $(q,p)$ is approximately the envelope of a family of segments whose endpoints are the B\'ezout coefficients of  coprime numbers belonging to neighborhoods of  $(p,q)$ and $(q,p)$, respectively. 
\end{abstract}

\keywords{B\'ezier quadratic curve, B\'ezout coefficients, coprime numbers, convex envelope.}

\msc{11H06 (Primary) 65D17; 11A05 (Secondary)}
\section*{Introduction}
B\'ezier curves are parametrized curves extensively used in computer aided geometric design (CAGD) \cite{Mortenson}. B\'ezier curves are related to Berenstein polynomials, they have very interesting mathematical properties and the literature contains many proposals of algorithms for their algebraic construction and their visualization \cite{Farin,winkel}. In this work, we will focus on quadratic B\'ezier curves. Since a quadratic B\'ezier curve  may be described as the envelope of a family of segments whose endpoints move along  two straight lines, then it is conceivable to replace this family of lines with another suitable family of lines whose corresponding endpoints are close to each other. When the quadratic B\'ezier curve is relative to the points $(p,q)$, $(0,0)$ and $(q,p)$, it turns out that such replacement lines may be characterized as having as endpoints pairs of B\'ezout coefficients of coprime numbers close to the points  $(p,q)$ and $(q,p)$, respectively. The main observation that leads to the verification of this statement  is fairly simple: Denote $P_0=(p,q)$, $P_1=(0,0)$ and $P_2=(q,p)$ and suppose that the coprime pair of positive numbers $(r,s)$ is close to the pair $(p,q)$. Then,  the B\'ezout coefficients $Q_1$ and $Q_2$ of the two coprime pairs  $(r,s)$ and $(s,r)$, respectively, are close to their projections $R_1$ and $R_2$ on the lines $P_0P_1$ and $P_1P_2$, respectively. Moreover, the distance from $R_1$ to $P_0$ is approximately the same as the distance from $R_2$ to $P_1$. Hence, the line with endpoints $Q_1$ and $Q_2$ is the candidate to replace a line in the envelope of the quadratic B\'ezier curve associated to $P_0$, $P_1$ and $P_2$. When the norm of $(p,q)$ is sufficiently large, there will be many coprime pairs in a neighborhood of $(p,q)$. Thus, when plotting all such lines with endpoints $Q_1$ and $Q_2$, they will approximate the quadratic B\'ezier curve corresponding to $P_0$, $P_1$ and $P_2$.

This work is divided in two sections. In Section 1 we establish the notation and state the elementary results needed in the rest of the paper and Section 2 contains the main results.

\section{B\'ezier Quadratics}

The linear B\'ezier curve for two points ${P}_0$ and ${P}_1$ in the plane, is defined as just the parametrized line $t\mapsto (1-t) {P}_0+t {P}_1$ from $P_0$ to $P_1$.
Given three different points $ {P}_0$, $ {P}_1$ and $ {P}_2$ in the plane, the quadratic B\'ezier curve is a parametrized curve in terms of the linear B\'ezier curves from $ {P}_0$ to $ {P}_1$ and from $ {P}_1$ to $ {P}_2$, that is,
the quadratic B\'ezier curve for the three points $ {P}_0$, $ {P}_1$ and ${P}_2$ is the parametric curve: 
\begin{align*}
t&\mapsto (1-t)\left((1-t) {P}_0+t {P}_1\right)+t\left((1-t) {P}_1+t {P}_2\right) 
=(1-t)^{2} {P} _{0}+2(1-t)t {P} _{1}+t^{2} {P} _{2}.
\end{align*}

\begin{definition}\label{Bezier}
Given a pair of nonnegative integers $(p,q)$, with $p\neq q$,  we will denote
by $\alpha_{p,q}$ and $\beta_{p,q}$ the linear B\'ezier curves  corresponding to the pairs of points  $(p,q)$, $(0,0)$ and $(0,0)$, $(q,p)$, respectively. That is
\begin{equation*}
    \alpha_{p,q}(t)=(1-t)(p,q)\hspace*{1cm}\text{ and }\hspace*{1cm}
    \beta_{p,q}(t)=t(q,p).
\end{equation*}
In addition, for each $0\leq s \leq 1$ we will denote by $\gamma_{s}$ the parametric line through the points $\alpha_{p,q}(s)$ and $\beta_{p,q}(s)$. Therefore
\begin{equation*}
    \gamma_s(t)=(1-t)\alpha_{p,q}(s)+t\beta_{p,q}(s).
\end{equation*}
The quadratic B\'ezier curve 
$c_{p,q}$ corresponding to  $P_0=(p,q)$, $P_1=(0,0)$ and $ P_2=(q,p)$, is then defined as
\begin{equation*}
    c_{p,q}(t)=(1-t)\alpha_{p,q}(t)+t\beta_{p,q}(t)= (1-t)^2(p,q)+t^2(q,p){\mbox{ , }}0\leq t\leq 1.
\end{equation*}
\end{definition}

For each $0< t_0< 1$, it is well known that a parametrization of the straight line tangent to the quadratic B\'ezier curve  $c_{p,q}(t)$ at the point $c_{p,q}(t_0)$, when $t=t_0$, is the line $\gamma_{t_0}$. Indeed, a parametrization of  such a tangent line is:
\begin{align*}
    c_{p,q}(t_0)+\dfrac{1}{2}(t-t_0)c^{\prime}_{p,q}(t_0)&= 
    c_{p,q}(t_0)+(t-t_0)\left(-(1-t_0)(p,q)+t_0(q,p)\right)\\
    &=(1-t_0)\left((1-t_0)-(t-t_0)\right)(p,q)+t_0\left(t_0+(t-t_0)\right)(q,p)\\
    &=(1-t)\gamma_{p,q}(t_0)+t\beta_{p,q}(t_0)\\
    &=\gamma_{t_0}(t).
\end{align*}

In fact, it is well known that the quadratic B\'ezier curve $c_{p,q}$ may be described as the envelope of the family of lines $\gamma_s$.

The following notion was introduced  as  Definition~1
in \cite{ILM}.
\begin{definition}\label{defbezout}
Given a pair of positive coprime integers $(p,q)$, we define the {\em B\'ezout coefficients} of $(p,q)$ 
as the unique pair of coprime numbers, denoted $\mathcal B(p,q)=(a,b)$, such that $0<a \leq p$, $0\leq b < q$  and  satisfy  {B\'ezout's identity}
\begin{equation*}
    aq-bp=1.
\end{equation*}

 \end{definition}

The following proposition gives a couple of formulas relating the B\'ezout coefficients of a coprime pair and the B\'ezout coefficients of the flipped pair. 

\begin{proposition}\label{prop}
Let $(p,q)$ be a coprime pair of nonnegative integers. If $\B(p,q)=(a,b)$ then following identities hold.
    \begin{enumerate} 
    \item  $\B(q,p)=(q-b,p-a)$.
    \item  $\B(b+q,a+p)=(q,p)$.
    \end{enumerate}
\end{proposition}

\begin{proof}
   Let $p\geq 1$ and $q\geq 1$ be relatively prime positive numbers. Let us assume $B(p,q)=(a,b)$, so that $aq-bp=1$, $0<a\leq p$ and $0\leq b<q$. For part $(1)$, we calculate
   \begin{align*}
       p(q-b)-q(p-a)&=aq-bp\\
                    &=1.
   \end{align*}
Therefore, as $0\leq p-a < p$ and $0< q-b\leq q$, we obtain 
\begin{equation*}
    \B(q,p) =(q-b,p-a),\\
\end{equation*}
as wanted.

To prove part $(2)$, we calculate
   \begin{align*}
       q(a+p)-p(b+q)&=aq-bp\\
                    &=1.
   \end{align*}
Therefore $p+a$ y  $b+q $ are relatively prime \cite[Theorem~5.1]{niven}. Since also the inequalities $0<q\leq b+q$ and $0\leq p \leq a+q$ hold,   we obtain
\begin{align*}
    \B(b+q,a+p)&=(q,p),\\
\end{align*}
as was to be proved.\qedhere
 
\end{proof}

We now prove a result that allows to argue why the points $\B(r,s)$ and $\B(s,r)$  are suitable substitutes for the endpoints of the envelope of the family of lines $\gamma_s$  of the quadratic B\'ezier curve $c_{p,q}$ relative to $(p,q)$, $(0,0)$ and $(q,p)$, when $(r,s)$ is close to $(p,q)$. More precisely, we prove that the distances from the projections of $\B(p,q)$ and $\B(q,p)$ on the lines $y=\frac{q}{p}x$ and $y=\frac{p}{q}x$ to $(p,q)$ and $(0,0)$, respectively, are equal. Furthermore, the distances from $\B(p,q)$ and $\B(q,p)$ to the lines $y=\dfrac{q}{p}x$ and $y=\dfrac{p}{q}x$ are both equal to the inverse of the norm of $(p,q)$.  It will be proved in the next section why such segments with endpoints $\B(r,s)$ and $\B(s,r)$  play the role of a segment $\gamma_{s}$ in defining the  quadratic B\'ezier curve $c_{p,q}$.

\begin{proposition}\label{mainprop}
Let $p$ and $q$ be relatively prime positive numbers. The following conditions hold.
   \begin{enumerate}
    \item  $\|\B(p,q)-(p,q)\|=\|\B(q,p)-(0,0)\|$.
    \item The distance from the point $\B(p,q)$ to the line $y=\dfrac{q}{p}x$  and the distance from the point $\B(q,p)$ to the line $y=\dfrac{p}{q}x$ are both  equal to $\dfrac{1}{\sqrt{p^2+q^2}}$. 
    \item Let $t_0=1-\dfrac{\B(p,q)\cdot (p,q)}{\|(p,q)\|^2}$. Then $0<t_0<1$ and the projection of the point $\B(p,q)$ on the line $y=\dfrac{q}{p}x$ is  $(1-t_0)(p,q)$, while the projection of the point $\B(q,p)$ on the line $y=\dfrac{p}{q}x$ is $t_0(q,p)$. 
    \item The distance from the point $(p,q)$ to the projection of $\B(p,q)$ on the line $y=\dfrac{q}{p}x$ is
    equal to the distance from the origin $(0,0)$ to the projection of $\B(q,p)$  on $y=\dfrac{p}{q}x$.
    \end{enumerate}
\end{proposition}
\begin{proof}
   
Let $p$ and $q$  be relatively prime positive integers. Assume $\B(p,q)=(a,b)$. From part (1) in Proposition~\ref{prop} we see that $\B(q,p)=(q-b,p-a)$. Hence
\begin{align*}
    \|\B(q,p) -(0,0)\| &= \sqrt{(q-b)^2+(p-a)^2}\\
                       &=\| (q,p)-(b,a)\|\\
                       &= \| (a,b)-(p,q)\|\\
                       &=\| \B(p,q)-(p,q)\|,
\end{align*}
thus proving (1).

For $(2)$, we compute the distance from $\B(p,q)=(a,b)$ to $y=\dfrac{q}{p}x$ to be 
\begin{align*}
    \dfrac{\left| aq-bp\right|}{\sqrt{p^2+q^2}} &= \dfrac{1}{\sqrt{p^2+q^2}}.
\end{align*}
 Similarly, if $\B(p,q)=(a,b)$ then the distance from $\B(q,p)=(q-b,p-a)$ to $y=\dfrac{p}{q}x$ is
\begin{align*}
    \dfrac{\left| p(q-b)-q(p-a)\right|}{\sqrt{p^2+q^2}} &= \dfrac{1}{\sqrt{p^2+q^2}}.
\end{align*}

Now, to prove (3), let us assume again that $\B(p,q)=(a,b)$.  Then, by Definition~\ref{defbezout}, we have the inequalities $0\leq a <p$ and $0<b\leq q$. 
We obtain that the vector resolution of $\B(p,q)$ in the direction of $(p,q)$ has length less than $\|(p,q)\|$, that is, $0< \dfrac{\B(p,q)\cdot (p,q)}{\|(p,q)\|}< \|(p,q)\|$. Thus $0< t_0 =1-\dfrac{\B(r,s)\cdot (r,s)}{\|(r,s)\|^2}< 1$. 

By Proposition~\ref{prop} we then have $\B(q,p)=(q-b,p-a)$. The coordinates of the projection of   $\B(p,q)$ on $y=\dfrac{q}{p}x$ are nothing but
\begin{equation*}
    \dfrac{\B(p,q)\cdot (p,q)}{\|(p,q)\|^2} (p,q)=(1-t_0)(p,q),
\end{equation*}
while the coordinates of the projection of $B(q,p)$ on $y=\dfrac{p}{q}x$ are
\begin{equation*}
    \dfrac{\B(q,p)\cdot (q,p)}{\|(q,p)\|^2} (q,p).
\end{equation*}
 By computing
 \begin{align*}
     t_0&=1-\dfrac{\B(p,q)\cdot (p,q)}{\|(p,q)\|^2} \\
     &=\dfrac{(p,q)\cdot (p,q)}{\|(p,q)\|^2}-\dfrac{(a,b)\cdot (p,q)}{\|(p,q)\|^2}\\
     &=\dfrac{(p-a,q-b)\cdot (p,q)}{\|(p,q)\|^2}\\
     &=\dfrac{\B(q,p)\cdot (q,p)}{\|(p,q)\|^2}\\
\end{align*}
 we complete the proof of $(3)$.

  Finally, for the proof of $(4)$, notice that the triangle with vertices $\B(p,q)$, $(p,q)$ and the point on the line $qx-py=0$ closest to $\B(p,q)$ is congruent to the  triangle with vertices $\B(q,p)$, $(0,0)$ and the point on the line $px-qy=0$ closest to $\B(q,p)$, by parts (1) and (2). Thus, (4) follows.  We may also obtain the result by direct computation using part $(3)$  as follows.
\begin{align*}
    \left\|\dfrac{\B(q,p)\cdot (q,p)}{\|(q,p)\|^2} (q,p) -(0,0)  \right\|
    &= \left\| t_0 (q,p)   \right\|  \\
    &=\left\| t_0 (p,q)   \right\|  \\
    &=\left\|
    \left(1-\dfrac{\B(p,q)\cdot (p,q)}{\|(p,q)\|^2}\right) (p,q)
    \right\|\\
    &=\left\|(p,q)-
    \dfrac{\B(p,q)\cdot (p,q)}{\|(p,q)\|^2} (p,q)
    \right\|,
\end{align*}
as was to be proved.
\end{proof}

Since given a coprime pair $(p,q)$ the line with endpoints $\B(p,q)$ and $\B(q,p)$ will play an important role in the sequel, we introduce a notation for it next.

\begin{definition}\label{bezoutlines}
For a pair $(p,q)$ of relatively prime positive numbers, we will denote by  $L_{p,q}$  the linear B\'ezier curve from $\B(p,q)$ to  $\B(q,p)$, that is, $L_{p,q}(t)=(1-t)\,\B(p,q)+t\, \B(q,p)$. We will call $L_{p,q}$ the B\'ezier-B\'ezout segment corresponding to $(p,q)$.
\end{definition}

Finally, we will define what we mean for two segments to be ``close''.

\begin{definition}\label{segmentdistance}
Let $L_i$ be a segment with endpoints at $A_i$ and $B_i$, for $i=1,2$. We 
define the distance from $L_1$ to $L_2$
as \[
\mathrm{dist}(L_1,L_2)=\max\bigl\{
        \min\{\|A_1-A_2\|,\|A_1-B_2\|\},
        \min\{\|B_1-A_2\|,\|B_1-B_2\|\}
        \bigr\}.
\]
\end{definition}

There exists a more precise notion of the distance between two segments, as the minimum distance between two points on each segment \cite{Allen}. However,  since we will measure the distance between segments which have close endpoints, it follows that with their natural parametrizations, all other pair of points in each segment remain close, as the following proposition shows. Thus, for our purposes, the formula for the distance between two segments, given in Definition~\ref{segmentdistance}, is good enough.

\begin{proposition}\label{points}
   Let $L_i$ be a segment in $\R^n$ with end points at $A_i$ and $B_i$, for $i=1,2$. Suppose that $\|A_1-A_2\|=\min\{ \|A_1-A_2\|, \|A_1-B_2\|\}$ and $\|B_1-B_2\|=\min\{ \|B_1-A_2\|,\|B_1-B_2\|\}$. Consider $\gamma_i(t)=(1-t)A_i+tB_i$  a parametrization of $L_i$. If  $\mathrm{dist}(L_1,L_2)<\epsilon$, for some $\epsilon>0$, then $\|\gamma_1(t)-\gamma_2(t)\|<\epsilon$, for every $0\leq t\leq 1$.
\end{proposition}
\begin{proof}
   Assume that $\mathrm{dist}(L_1,L_2)<\epsilon$. Let $0\leq t\leq 1$. Then
   \begin{align*}
    \|\gamma_1(t)-\gamma_2(t)\|&\leq
    (1-t)\|A_1-A_2\|+t\|B_1-B_2\|\\
    &=(1-t)\min\{ \|A_1-A_2\|, \|A_1-B_2\|\} +
      t\min\{ \|B_1-A_2\|,\|B_1-B_2\|\}\\
      &\leq (1-t)\mathrm{dist}(L_1,L_2)
           +t\, \mathrm{dist}(L_1,L_2)\\
      &<\epsilon.
  \end{align*}
\end{proof}

\section{Approximating a quadratic B\'ezier curve}

In this section we prove our main result, namely, that a quadratic B\'ezier curve $c_{p,q}$ as in Definition~\ref{Bezier}, characterized as the envelope of the family of lines $\gamma_{s}$, can be approximated through a family of  B\'ezier-B\'ezout line segments given
in Definition~\ref{bezoutlines}.

For integers $p>3$ and $0\leq q <p$, our next proposition will show how for a given $1\leq\epsilon\leq\frac{1}{2}\|(p,q)\|$ and any pair of relatively prime numbers $(r,s)$ in an $\epsilon$-ball centered at $(p,q)$,  one gets a real number $0< t_0< 1$ such that both the distances from $\B(r,s)$  to $\alpha_{p,q}(t_0)$ and  from $\B(s,r)$ to $\beta_{p,q}(t_0)$ are less to $\epsilon +1$. Since $\B(r,s)$ and $\B(s,r)$ are the endpoints of the B\'ezier-B\'ezout line $L_{r,s}$, while $\alpha_{p,q}(t_0)$ and $\beta_{p,q}(t_0)$ are the endpoints of $\gamma_{t_0}$, this gives that  $\mathrm{dist}(L_{r,s},\gamma_{t_0})<\epsilon+1$.  Thus, the next proposition will prove to  be crucial for our main result.

\begin{proposition}\label{aprox}
   Consider integers $p\geq 2$ and $0\leq q$  
   and let $1\leq \epsilon\leq \|(p,q)\|$ be a real number. If $(r,s)$ is a pair of positive coprime numbers and  $\|(r,s)-(p,q)\|\leq \epsilon$,  
   then for $t_{0}=1-\dfrac{\B(r,s)\cdot (r,s)}{\|(r,s)\|^2}$, one has that $0<t_0<1$. Furthermore
   \[
   \|\B(r,s)-\alpha_{p,q}(t_{0})\|< \epsilon+1 \hspace*{0.3cm}\text{ and }\hspace*{0.3cm}
   \|\B(s,r)-\beta_{p,q}(t_{0})\|<
    \epsilon+1,
   \]
   where  $\alpha_{p,q}(t)=(1-t)(p,q)$ and $\beta_{p,q}(t)=t(q,p)$ are the parametrized lines given in Definition~\ref{Bezier}. 
\end{proposition}
\begin{proof}
Let $r$ and $s$ be coprime positive integers such that  $\|(r,s)-(p,q)\|\leq \epsilon$. By Proposition~\ref{mainprop}~(3), it follows that $0<t_0<1$. On the other hand,  
by Proposition~\ref{mainprop}~(2) and (3), the distance from $\B(r,s)$ to the line $y=\frac{s}{r}x$  is equal to the distance from $\B(r,s)$ to the point $(1-t_0)(r,s)$. Therefore $\|\B(r,s)-(1-t_0)(r,s)\|=\dfrac{1}{\sqrt{r^2+s^2}}$. Thus
\begin{align*}
    \|\B(r,s)-\alpha_{p,q}(t_0)\| &=
    \| \B(r,s)-(1-t_0)(p,q)\|\\
    &\leq \|\B(r,s)-(1-t_0)(r,s)\| + \|(1-t_0)(r,s)-(1-t_0)(p,q)\|\\
    &=\dfrac{1}{\sqrt{r^2+s^2}}+(1-t_0)\, \|(r,s)-(p,q)\|\\
    &< \dfrac{1}{\sqrt{p^2+q^2}-\epsilon} + \epsilon\\
    &\leq \epsilon +1
\end{align*}

On the other hand, using again Proposition~\ref{mainprop} (2) and (3), the distance from $\B(s,r)$ to the point $t_0(s,r)$ is equal to the distance form the point $\B(s,r)$ to the line $\frac{s}{r}$. Therefore we  get  $\left\|\B(s,r)-t_0   (s,r)\right\|=\dfrac{1}{\sqrt{r^2+s^2}}$ . Thus
\begin{align*}
    \|\B(s,r)-\beta_{p,q}(t_0)\| &=
    \| \B(s,r)-t_0(q,p)\|\\&\leq \|\B(s,r)-t_0(s,r)\| + \|t_0(s,r)-t_0(q,p)\|\\
    &=\left\|\B(s,r)-t_0 (s,r)\right\| +t_0\, \|(s,r)-(q,p)\|\\
    &=\frac{1}{\sqrt{r^2+s^2}}+t_0\, \|(s,r)-(q,p)\|\\
    &< \dfrac{1}{\sqrt{p^2+q^2}-\epsilon} + \epsilon\\
    &\leq \epsilon +1
\end{align*}

\end{proof}

The following theorem is our main result. It establishes the approximation of the B\'ezier quadratic curve $c_{p,q}$  in Definition~\ref{Bezier} through segments having as endpoints pairs of B\'ezout coefficients. More precisely, it establishes conditions for a point in  the  B\'ezier-B\'ezout segment  $L_{r,s}$ to be close to a point of the quadratic B\'ezier  curve $c_{p,q}$.

\begin{theorem}\label{main}
 Let $p$ and $q$ be nonnegative integers with $p\not=q$ and $p\geq 2$. 
  Let  $2\leq  \epsilon+1\leq  \|(p,q)\|$ be a real number. Consider $c_{p,q}$ the quadratic B\'ezier curve corresponding to $P_0=(p,q)$, $P_1=(0,0)$ and $P_2=(q,p)$ defined in \ref{Bezier}. If $(r,s)$ is a pair of positive coprime numbers and   $\|(r,s)-(p,q)\|\leq \epsilon$  then there exits $0< t_{r,s}< 1$  such that 
 \begin{equation*}
   \mathrm{dist}(L_{r,s},\gamma_{t_0})<\epsilon+1 \quad\text{and}\quad\|L_{r,s}(t_{r,s})- c_{p,q}(t_{r,s})\|<\epsilon+1,
 \end{equation*}
where $L_{r,s}(t)$ is the B\'ezier-B\'ezout segment corresponding to $(r,s)$ defined in \ref{bezoutlines} and $\gamma_{t_0}$ is a parametrization of the straight line tangent to the quadratic B\'ezier curve  $c_{p,q}(t)$ at the point $c_{p,q}(t_0)$ as defined in \ref{Bezier}.

\end{theorem}
\begin{proof}
    Suppose that $r$ and $s$ are relatively prime numbers such that $\| (r,s)-(p,q)\|\leq \epsilon$. By Proposition~\ref{aprox}, there exists $t_{r,s}$ between 0 and 1 such that the distances from the points $\B(r,s)$ and $\B(s,r)$ to $\alpha_{p,q}(t_{r,s})$ and $\beta_{p,q}(t_{r,s})$, respectively, are both less than $\epsilon$+1. But the B\'ezier-B\'ezout line $L_{r,s}$ has end points $\B(r,s)$ and $\B(s,r)$, while the line $\gamma_{t_{r,s}}$ with endpoints on $\alpha_{p,q}(t_{r,s})$ and $\beta_{p,q}(t_{r,s})$ is tangent to the quadratic B\'ezier curve $c_{p,q}$ at $t=t_{r,s}$. So we conclude that $\mathrm{dist}(L_{r,s},\gamma_{t_{r,s}})<\epsilon+1$. The result now follows from Proposition~\ref{points}.
\end{proof}

In Figures~\ref{fig:envelope1} and \ref{fig:envelope2}, we show  some examples of approximately quadratic B\'ezier curves using $\epsilon=10$. 

\begin{figure}[ht!]
    \centering
\begin{tabular}{c @{\qquad} c }
        \includegraphics[width=60mm,scale=0.25]{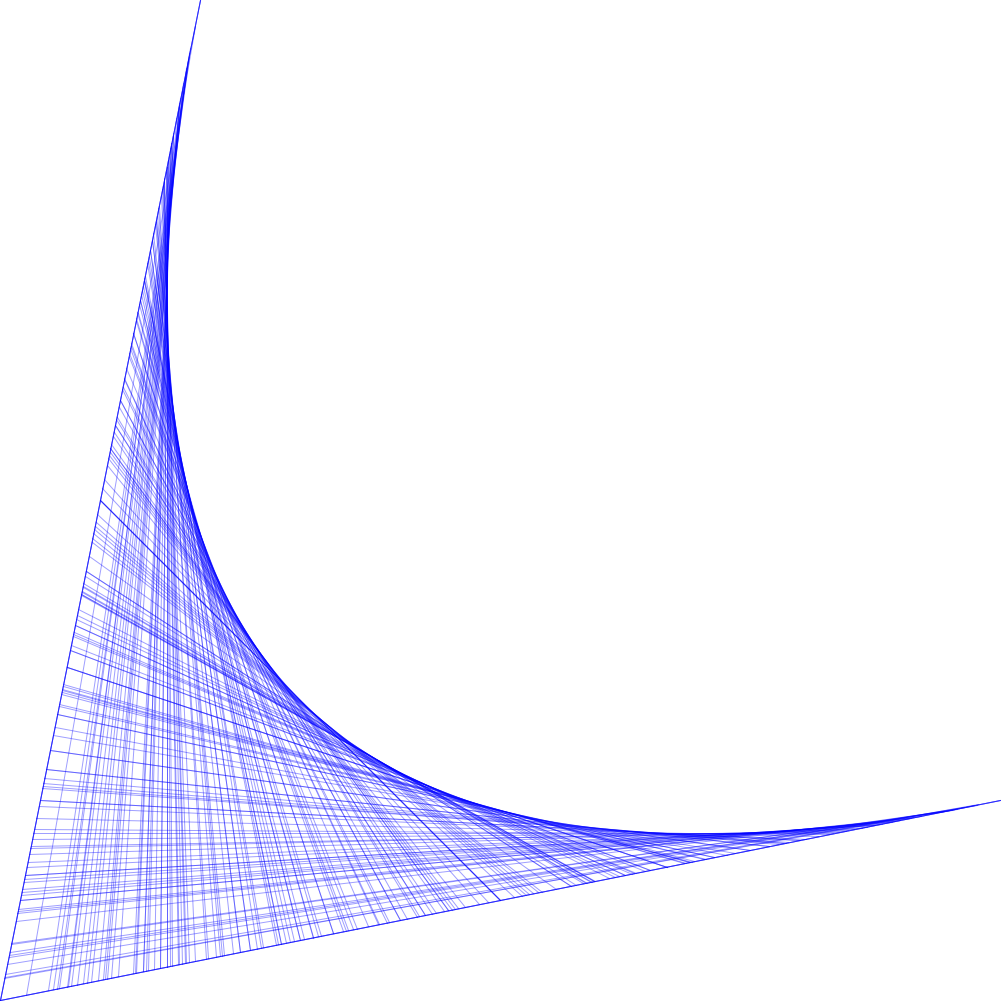}
      & \includegraphics[width=60mm,scale=0.25]{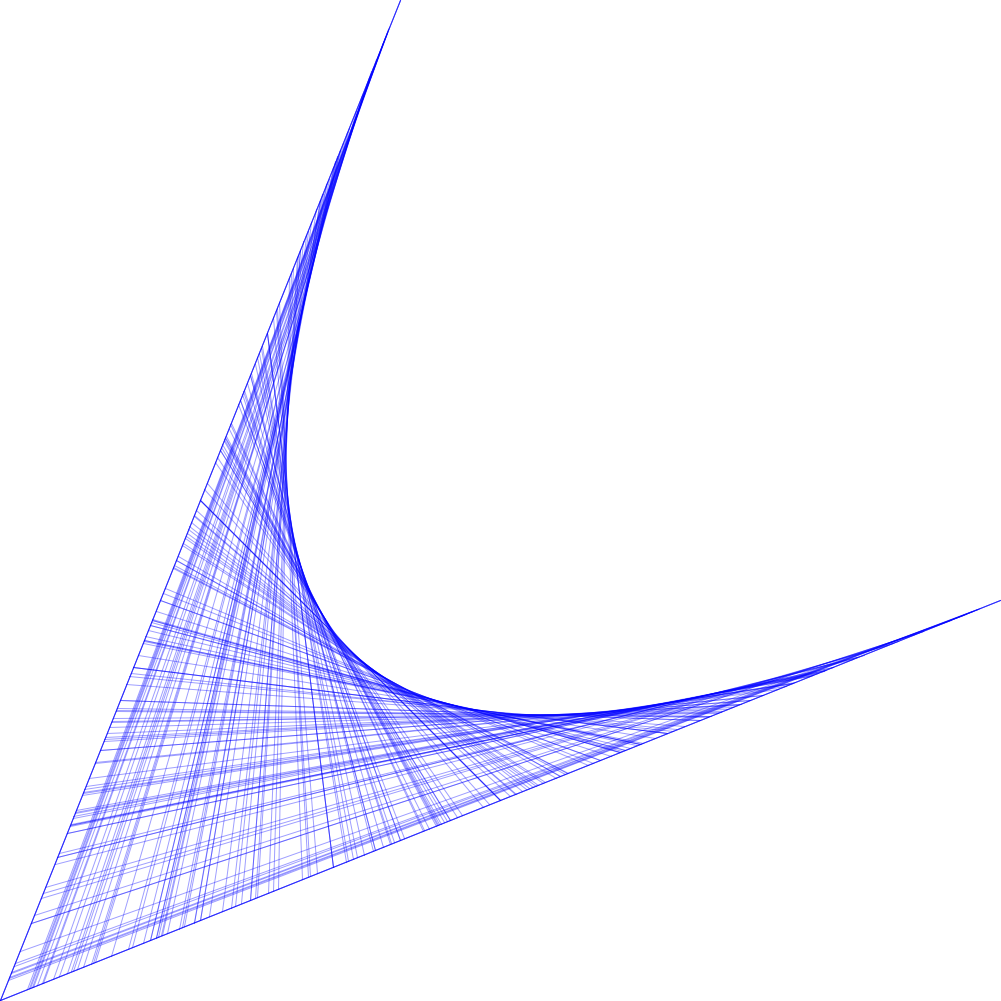}\\
\end{tabular}
    \caption{Approximate B\'ezier quadratic curve $c_{p,q}$ for $(p,q)=(1000000,200000)$ and $\epsilon=10$ (left) and $(p,q)=(1000000,400000)$ and $\epsilon=10$ (right). }
    \label{fig:envelope1}
\end{figure}

\begin{figure}[ht!]
    \centering
\begin{tabular}{c @{\qquad} c }
        \includegraphics[width=60mm,scale=0.25]{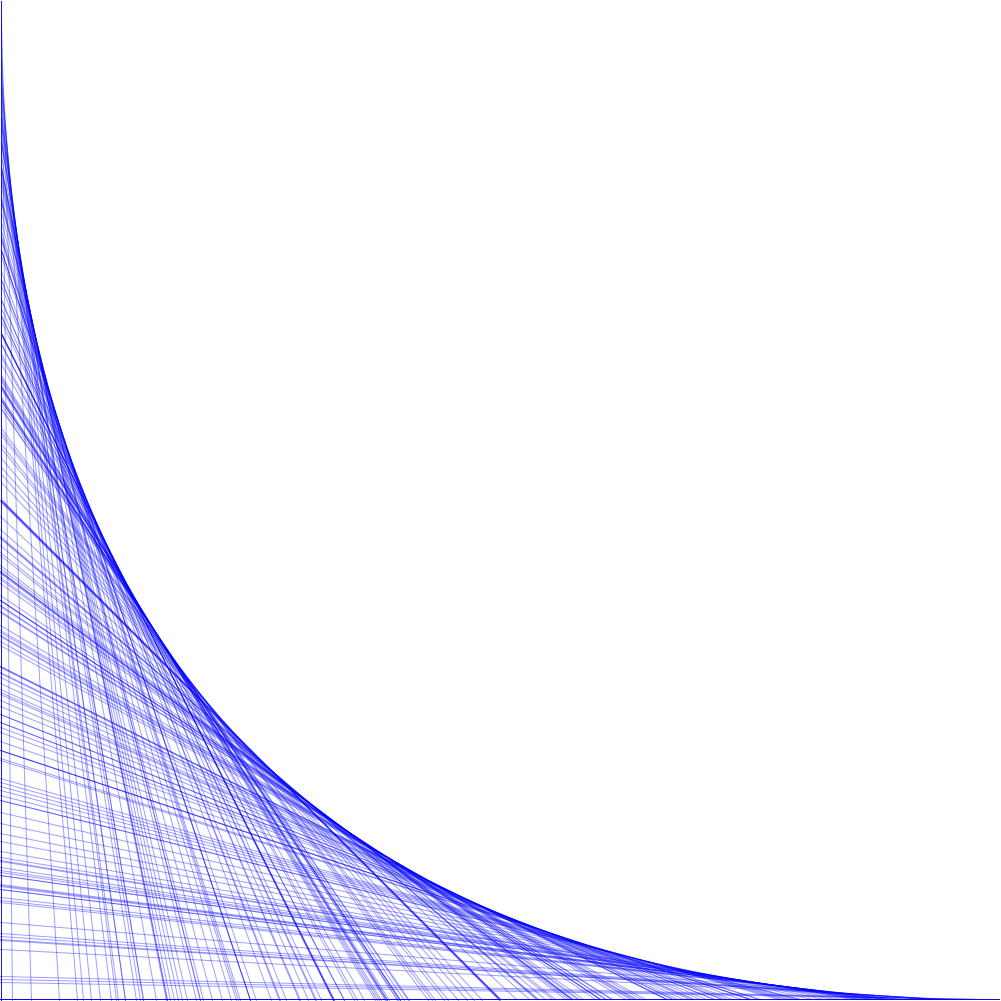}
      & \includegraphics[width=60mm,scale=0.25]{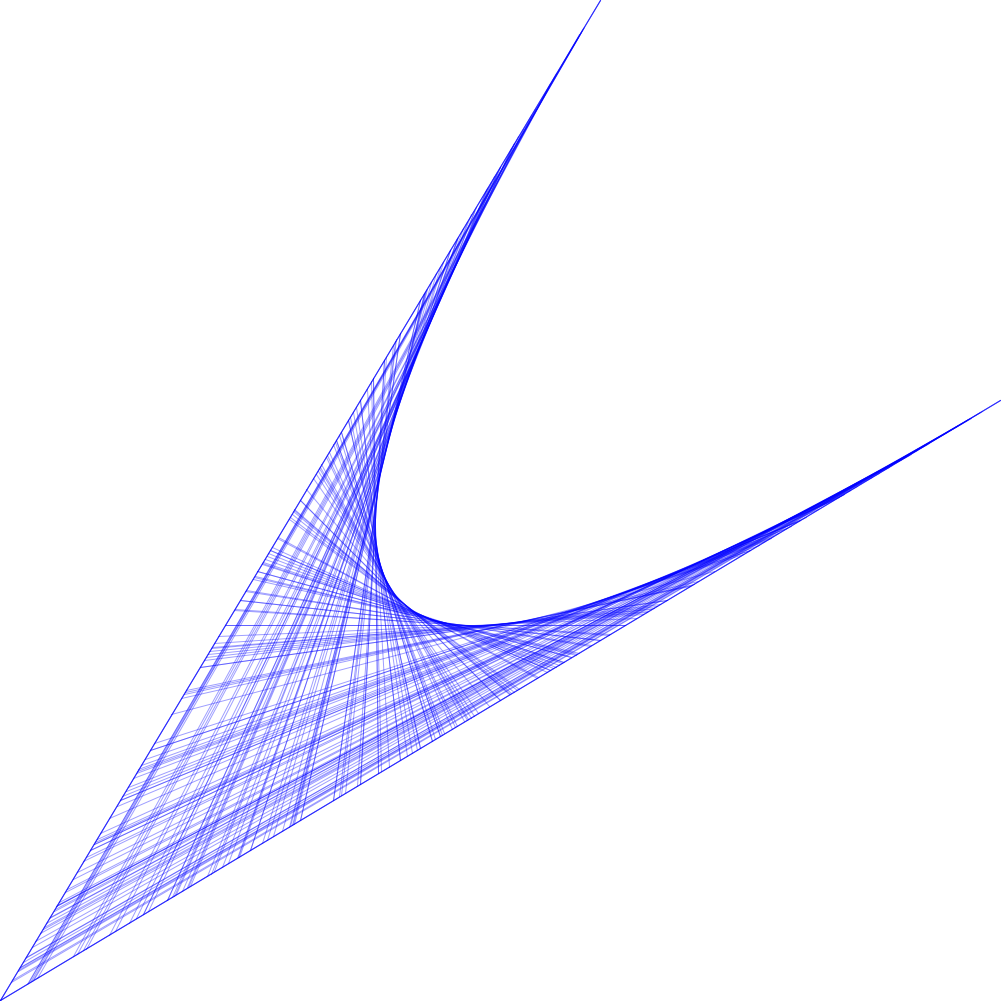}\\
\end{tabular}
    \caption{Approximate B\'ezier quadratic curve $c_{p,q}$ for $(p,q)=(1000000,1000)$ and $\epsilon=10$ (left) and $(p,q)=(1000000,600000)$ and $\epsilon=10$ (right). }
    \label{fig:envelope2}
\end{figure}

Let $V_{p,q}(\epsilon)$ denote the number of coprime numbers in the closed ball with center $(p,q)$ and radius $\epsilon$, that is to say, $V_{p,q}(\epsilon )$ is the cardinality of the set containing the coprime pairs $(r,s)$ such that $(r-p)^2+(s-q)^2\leq \epsilon^2$. It is known that when $(p,q)=(0,0)$, the value of $V_{p,q}(\epsilon)$ es approximately $\dfrac{6}{\pi}\epsilon^2$. This problem is known as  the Primitive Circle Problem \cite{zhai,wu}.  We observe that the value of $V_{p,q}(10)$ is approximately independent of the  integer coordinates $(p,q)$ for relatively large values of $p$ and $q$.  Then $V_{p,q}(10)\approx 190$. Hence, the number of B\'ezout-B\'ezier lines drawn in Figures~\ref{fig:envelope1} and \ref{fig:envelope2}, approximating the envelope of the B\'ezier quadratic curves is, in each case, approximately $190$. 
It is clear that if the value of $\epsilon$ is too small  then too few lines may be obtained (for example, for $p=300$,  $q=21$ and $\epsilon=1$, then $(299,21)$ is the only pair of coprime integers within $\epsilon$ of $(p,q)$). On the other hand, if the size of $\epsilon$ is increased too much to obtain more coprime pairs within $\epsilon$ of $(p,q)$, it may distort our approximation of the B\'ezier quadratic curve. In Figure~\ref{fig:envelope3}
we illustrate cases for too small and too big values of $\epsilon$.  

\begin{figure}[ht!]
    \centering
\begin{tabular}{c @{\qquad} c }
        \includegraphics[width=65mm,scale=0.3]{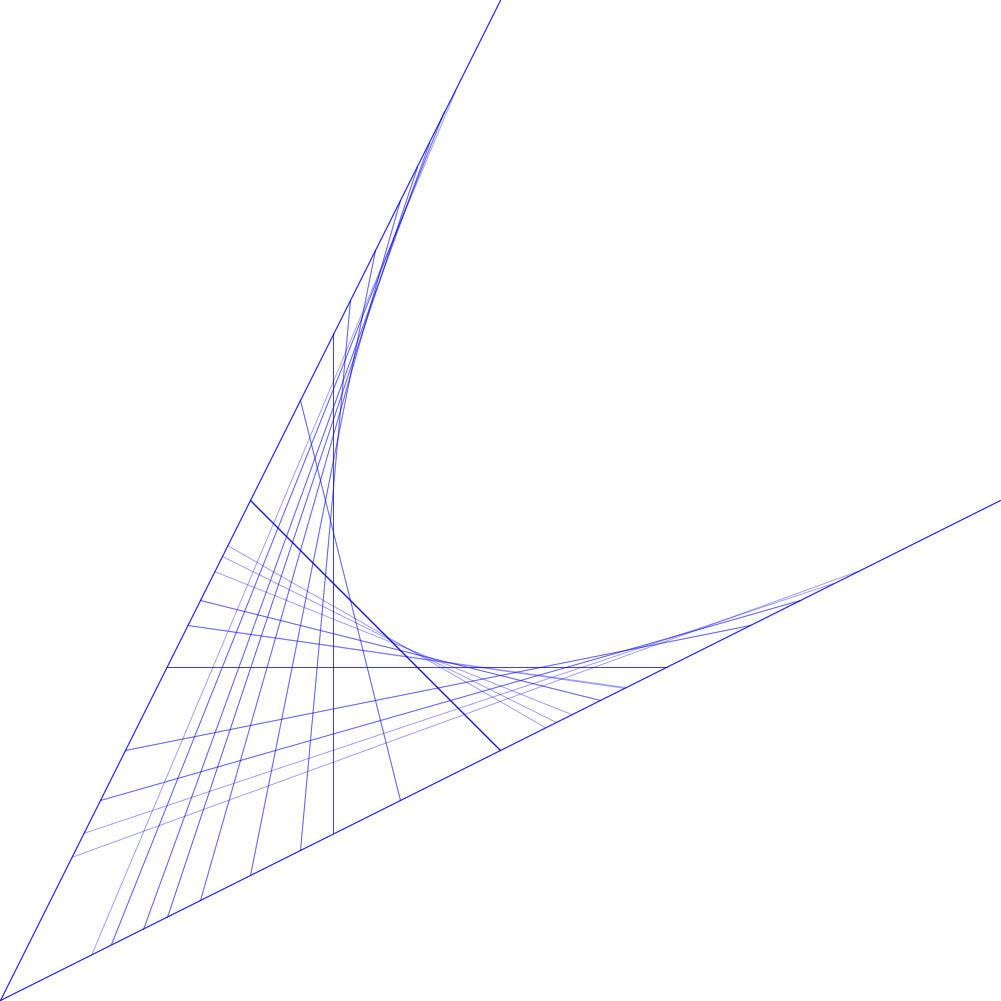}
      & \includegraphics[width=65mm,scale=0.3]{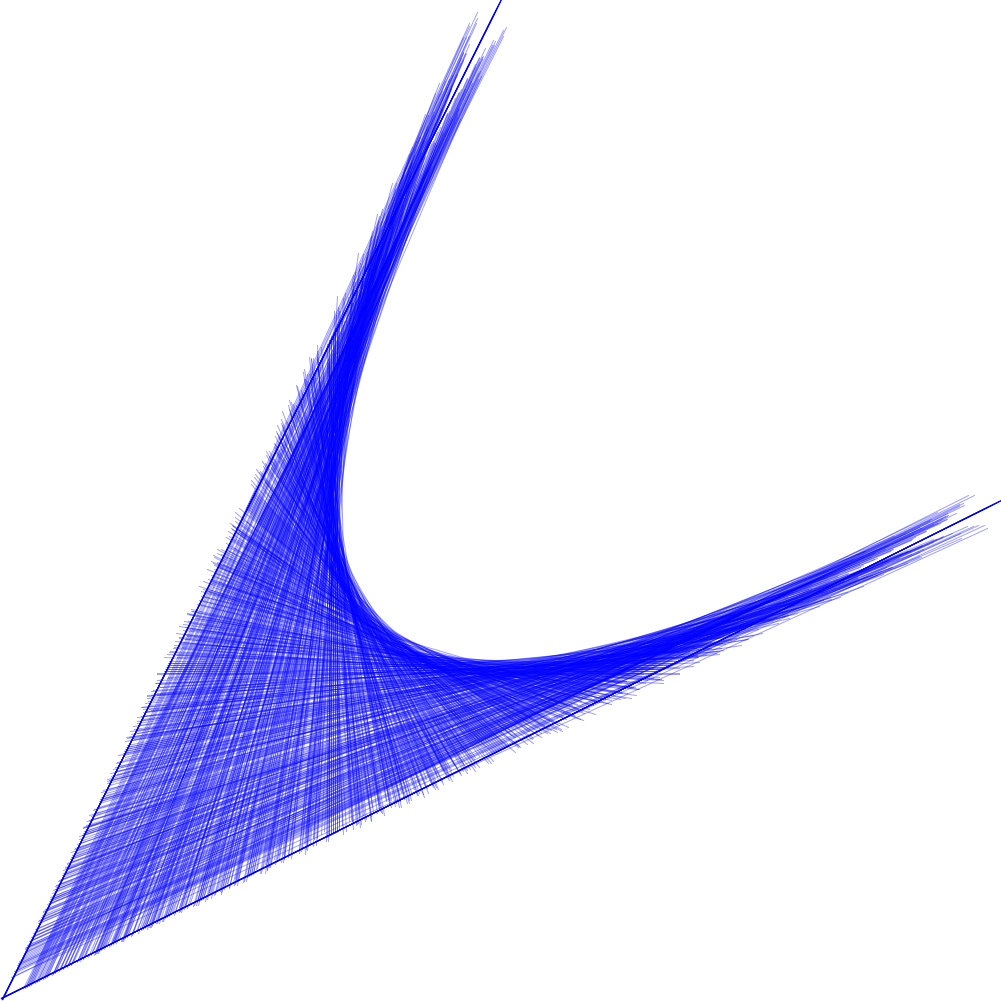}\\
\end{tabular}
    \caption{Distorted approximations of B\'ezier quadratic curves $c_{p,q}$ for $(p,q)=(1000000,500000)$ and $\epsilon=5$ (left, $\epsilon$ too small) and $(p,q)=(1000,500)$ and $\epsilon=20$ (right, $\epsilon$ too big).  }
    \label{fig:envelope3}
\end{figure}


\end{document}